\documentclass{amsart}

\usepackage{color}
\usepackage{mathrsfs}
\usepackage{amssymb}
\usepackage{graphicx}
\usepackage{wrapfig}

\newcommand{\samir}[1]{\textcolor{red}{Samir says: {#1}}}

\newcommand{\lp}{\left(}
\newcommand{\rp}{\right)}

\newcommand{\matele}[1]{({\!}({#1})\!)}  
\newcommand\mattwo[4]{\left(\begin{smallmatrix}
			{#1} & {#2}\\
     			{#3} & {#4}
                     \end{smallmatrix}\right)}
\newcommand\mattres[9]{\left(\begin{smallmatrix}
			{#1} & {#2}&{#3}\\
     			{#4} & {#5}&{#6}\\
     				{#7} &{#8} &{#9}
                     \end{smallmatrix}\right)}

\newcommand{\ds}{\displaystyle}
\renewcommand{\r}{\rightarrow}
\newcommand{\N}{\mathbb{N}}

\newcommand{\R}{\mathbb{R}}

\newcommand{\e}{\varepsilon}
\newcommand{\set}[1]{\left\{#1\right\}}
\newcommand{\diam}{\operatorname{diam}}

\newcommand{\dis}{\operatorname{dis}}		
\newcommand{\dgh}{d_{\operatorname{GH}}}		
\newcommand{\sep}{\operatorname{sep}}
\renewcommand{\H}{\operatorname{H}}		
\newcommand{\M}{\mathcal{M}}		
\newcommand{\Msim}{\mathcal{M}/\!\!\sim}
\renewcommand{\dh}{\delta_{\operatorname{H}}}
\newcommand{\Rsc}{\mathscr{R}}		
\renewcommand{\a}{\alpha}
\renewcommand{\b}{\beta}
\newcommand{\g}{\gamma}
\renewcommand{\d}{\delta}

\newcommand{\ph}{\varphi}
\renewcommand{\t}{\tau}

\newcommand{\Ropt}{\Rsc^{\operatorname{opt}}}

\newcommand{\mc}[1]{\mathcal{#1}}
\newcommand{\sph}{\mathbb{S}}
\newcommand{\diag}{\blacktriangle}

\newtheorem{proposition}{Proposition}[section]
\newtheorem{theorem}{Theorem}[section]
\newtheorem{lemma}[theorem]{Lemma}

\theoremstyle{definition}

\newtheorem{example}[theorem]{Example}

\newtheorem{claim}[theorem]{Claim}

\theoremstyle{remark}
\newtheorem{remark}[theorem]{Remark}

\numberwithin{equation}{section}

\begin{document}

\title{ Explicit Geodesics in Gromov-Hausdorff space}

\author{Samir Chowdhury}
\address{Department of Mathematics,
The Ohio State University, 
100 Math Tower,
231 West 18th Avenue, 
Columbus, OH 43210. 
Phone: (614) 292-4975,
Fax: (614) 292-1479 }

\email{chowdhury.57@osu.edu}
\thanks{}

\author{Facundo M\'emoli}
\address{Department of Mathematics,
The Ohio State University, 
100 Math Tower,
231 West 18th Avenue, 
Columbus, OH 43210. 
Phone: (614) 292-4975,
Fax: (614) 292-1479 }
\curraddr{}
\email{memoli@math.osu.edu}
\thanks{This work was supported by NSF grants CCF-1526513 and IIS-1422400}

\subjclass[2010]{Primary 53C23, Secondary 51F99}

\date{}



\begin{abstract}
We provide an alternative, constructive proof that the collection $\M$ of isometry classes of compact metric spaces endowed with the Gromov-Hausdorff distance is a geodesic space. The core of our proof is a construction of explicit geodesics on $\M$. We also provide several interesting examples of geodesics on $\M$, including a geodesic between $\sph^0$ and $\sph^n$ for any $n\geq 1$. 

\end{abstract}

\maketitle

\section{Geodesics on Gromov-Hausdorff space}

The collection of compact metric spaces, denoted $\M$ throughout this paper, is a valid pseudometric space when endowed with the Gromov-Hausdorff distance \cite{gromov-book, burago}. We will denote this space by $(\M,\dgh)$. Furthermore, the space $(\Msim,\dgh)$, where we define
\begin{align*}
(X,d_X)\sim (Y,d_Y) &\iff \text{$(X,d_X)$ is isometric to $(Y,d_Y)$, and}\\
\dgh([X],[Y]) &:= \dgh(X,Y),
\end{align*}
is a metric space \cite{burago}. It is known that $(\Msim,\dgh)$ is separable, complete \cite{petersen2006riemannian}, and geodesic \cite{ivanov2015gromov}. Specifically, the authors of \cite{ivanov2015gromov} use a compactness result to argue that for any pair of points in $\Msim$, there exists a \emph{midpoint} in $\Msim$, which implies that $(\Msim,\dgh)$ is a geodesic space \cite[Theorem 2.4.16]{burago}. 
However, this proof is not constructive. The goal of our paper is to provide a constructive proof through the \emph{explicit description} of a certain class of geodesics on $(\Msim,\dgh)$, which we call \emph{straight-line geodesics}. The key ingredient in our construction is a proof showing that there exists an \emph{optimal correspondence} between any two compact metric spaces. 
While obvious when considering finite metric spaces, establishing this result for general compact metric spaces requires some work. Our result is inspired by a similar result proved by Sturm about geodesics on the space of metric measure spaces \cite{sturm2012space}. We use our result to construct: (1) an explicit geodesic between $\sph^0$ and $\sph^n$, for any $n\in \N$, and (2) explicit, infinite families of \emph{deviant} (i.e. non-straight-line) and branching geodesics between the one-point discrete space and the $n$-point discrete space, for any $n \geq 2$.

Before proceeding, we recall some concepts, in particular that of a \emph{geodesic space}. A \emph{curve} in a metric space $(X,d_X)$ is a continuous map $\g : [0,1] \r X$, and its \emph{length} is given by:
$$L(\g):= \sup\set{\sum_{i=1}^{n-1}d_X(\g(t_i),\g(t_{i+1}) : 0 = t_1 \leq t_2 \leq \ldots \leq t_n = 1, n\in \N}.$$
Such a curve is called a \emph{geodesic} \cite[Section I.1]{bridson2011metric} if for any $s, t \in [0,1]$, 
$$d_X(\g(s),\g(t)) = |t-s|\cdot d_X(\g(0),\g(1)).$$
As a consequence of this definition, for any geodesic $\g$ such that $\gamma(0) = x$ and $\gamma(1)=x'$, one has $L(\g) = d_X(x,x')$. The metric space $(X,d_X)$ is called a \emph{geodesic space} if for any $x,x' \in X$, there exists a geodesic $\g$ connecting $x$ and $x'$.

Next, given a metric space $(X,d_X)$ and two nonempty subsets $A,B \subseteq X$, the \emph{Hausdorff distance} between $A$ and $B$ is defined as:
\[d^X_{\H}(A,B):= \max\bigg(\sup_{a\in A}\inf_{b\in B}d_X(a,b), \sup_{b\in B}\inf_{a\in A}d_X(a,b)\bigg).\]

Given $(X,d_X), (Y,d_Y) \in \M$, the Gromov-Hausdorff distance between them is defined as:
\begin{align}\label{eq:dgh-defn}
\dgh((X,d_X),(Y,d_Y)):=\inf_{\substack{(Z,d_Z)\in \M\\
\ph:X\r Z, \,
\psi:Y \r Z}} d_{\H}^Z(\ph(X),\psi(Y)),
\end{align}
where $\ph$ and $\psi$ are both isometric embeddings \cite{burago}.  Notice that $\dgh$ is well-defined on $[X],[Y] \in \Msim$. Indeed, if $X' \in [X], Y' \in [Y]$, then:
$$\dgh([X'],[Y'])=\dgh(X',Y')= \dgh(X,Y)=\dgh([X],[Y]),$$
where the second-to-last equality follows from the triangle inequality and the observation that $\dgh(X,X') = \dgh(Y,Y') = 0$.


It is known that $(\Msim,\dgh)$ is complete \cite{petersen2006riemannian}. Details about the topology generated by the Gromov-Hausdorff distance can be found in \cite{burago}. One important fact is that it allows the existence of many compact sets in $(\Msim,\dgh)$, in the sense below. Recall that for a compact metric space $X$, for $\e>0$,  the \emph{$\e$-covering number} $\operatorname{cov}_X(\e)$ is defined to be the minimum number of $\e$-balls required to cover $X$. 


\begin{theorem}[Gromov's precompactness theorem, \cite{petersen2006riemannian}] Given a bounded function $N:(0,\infty) \r \N$ and $D > 0$, let $\mathcal{C}(N,D) \subseteq (\Msim,\dgh)$ be the collection of all $[X]$ such that $\diam(X) < D$ and $\operatorname{cov}_X(\e) \leq N(\e)$ for each $\e >0$. Then $\mathcal{C}(N,D)$ is precompact. 
\end{theorem}

In our constructions, we will use an equivalent formulation of the Gromov-Hausdorff distance following \cite[Chapter 7]{burago}. Given $(X,d_X), (Y,d_Y) \in \M$, we say that a relation $R\subseteq X \times Y$ is a \emph{correspondence} if for any $x \in X$, there exists $y\in Y$ such that $(x,y) \in R$, and for any $y\in Y$, there exists $x \in X$ such that $(x,y) \in R$. The set of all such correspondences will be denoted $\Rsc(X,Y)$. In the case $Y = X$, a particularly useful correspondence is the \emph{diagonal correspondence} $\diag := \{(x,x) : x \in X\}$. The \emph{distortion} of any non-empty relation  $R\subset X\times Y $ is defined to be:
\[\dis(R) := \sup_{(x,y),(x',y')\in R}|d_X(x,x')-d_Y(y,y')|.\]

The Gromov-Hausdorff distance $\dgh: \M \times \M \r \R_+$ can be formulated as:
\[\dgh((X,d_X),(Y,d_Y)) :=\frac{1}{2} \inf_{R \in \Rsc(X,Y)}\dis(R).\]

In particular, a correspondence is \emph{optimal} if the infimum is achieved. We will denote by $\Ropt(X,Y)$ the set of all closed optimal correspondences. We have:
\begin{proposition}\label{prop:optimal} $\Ropt(X,Y)\neq \emptyset$ for any $(X,d_X)$ and $(Y,d_Y)\in \M$. 
\end{proposition}

Our main result is the explicit construction of \emph{straight-line geodesics}:

\begin{theorem}[Existence of straight-line geodesics]
\label{thm:geodesic} $(\Msim, \dgh)$ is a geodesic space. More specifically, let $[X],[Y] \in (\Msim, \dgh)$. Then, for any $R\in \Ropt(X,Y)$, we can construct a geodesic $\g_R:[0,1] \r \Msim$ between $[X]$ and $[Y]$ as follows: 
\begin{align*}
&\g_R(0):=[(X,d_X)],\, \g_R(1):=[(Y,d_Y)], \text{ and }\g_R(t):=[(R,d_{\g_R(t)})] \text{ for } t\in (0,1),
\intertext{ where for each $(x,y),(x',y') \in R$ and $t\in(0,1),$} 
&\hspace{0.2\textwidth} d_{\g_R(t)}\big((x,y),(x',y')\big):=(1-t)\cdot d_X(x,x')+t\cdot d_Y(y,y').
\end{align*}
\end{theorem}

Not all geodesics between compact metric spaces are of the form given by Theorem \ref{thm:geodesic}. Furthermore, \emph{branching} of geodesics may happen in Gromov-Hausdorff space. We explore the deviance and branching phenomena in Section \ref{sec:geods}. In Section \ref{sec:spheres} we construct explicit geodesics between $\sph^0$ and $\sph^n$. The proofs of Proposition \ref{prop:optimal} and Theorem \ref{thm:geodesic} are given in Section \ref{sec:proofs}.

\subsection{Deviant and branching geodesics}\label{sec:geods}
The following lemma will be useful in the sequel:
\begin{lemma}\label{lemma:curves}
Let $(Z,d_Z)$ be a metric space. Let $S,T\in \mathbb{R}$, with $S<T$, and $\gamma:[S,T]\rightarrow Z$ be a curve such that 
\[d_Z(\gamma(s),\gamma(t))\leq \frac{|s-t|}{|S-T|}\cdot d_Z(\gamma(S),\gamma(T)),\,\,\,\mbox{for all $s,t\in[S,T]$}.\]
Then, in fact, 
\[d_Z(\gamma(s),\gamma(t))= \frac{|s-t|}{|S-T|}\cdot d_Z(\gamma(S),\gamma(T)),\,\,\,\mbox{for all $s,t\in[S,T]$}.\]
\end{lemma}
\begin{proof}[Proof of Lemma \ref{lemma:curves}] Suppose the inequality is strict. Suppose also that $s\leq t$. Then by the triangle inequality, we obtain:
\begin{align*}
d_Z(\g(S),\g(T)) &\leq d_Z(\g(S),\g(s)) + d_Z(\g(s),\g(t)) + d_Z(\g(t),\g(T))\\
&< \frac{(s-S) + (t-s) + (T-t)}{T-S}\cdot d_Z(\g(S),\g(T)).
\end{align*}
This is a contradiction. Similarly we get a contradiction for the case $t < s$. This proves the lemma. \end{proof}



\subsubsection{Deviant geodesics}
\label{sec:deviant}
For any $n \in \N$, let $\Delta_n$ denote the $n$-point discrete space, often called the $n$-point \emph{unit simplex}. Fix $n \in \N$, $n \geq 2$. We will construct an infinite family of \emph{deviant geodesics} between $\Delta_1$ and $\Delta_n$, named as such because they deviate from the straight-line geodesics given by Theorem \ref{thm:geodesic}. 
As a preliminary step, we describe the straight-line geodesic between $\Delta_1$ and $\Delta_n$ of the form given by Theorem \ref{thm:geodesic}. Let $\{p\}$ and $\{x_1,\ldots, x_n\}$ denote the underlying sets of $\Delta_1$ and $\Delta_n$. 
There is a unique correspondence $R:=\{(p,x_1),\ldots, (p,x_n)\}$ between these two sets. 
According to the setup in Theorem \ref{thm:geodesic}, the straight-line geodesic between $\Delta_1$ and $\Delta_n$ is then given by the metric spaces $(R,d_{\g_R(t)})$, for $t\in (0,1)$. Here $d_{\g_R(t)}((p,x_i),(p,x_j)) = t\cdot d_{\Delta_n}(x_i,x_j) = t$ for each $t \in (0,1)$ and each $1\leq i,j \leq n$. This corresponds to the all-$t$ matrix with $0$s on the diagonal.  
Finally, we note that the unique correspondence $R$ necessarily has distortion $1$. Thus $\dgh(\Delta_1,\Delta_n) = \tfrac{1}{2}$. 


Now we give the parameters for the construction of a certain family of deviant geodesics between $\Delta_1$ and $\Delta_n$. 
For any $\a \in (0,1]$ and $t \in [0,1]$, define 
\[f(\a,t):= \begin{cases}
t\a &: 0 \leq t \leq \tfrac{1}{2} \\
\a - t\a &: \tfrac{1}{2} < t \leq 1
\end{cases}\]
Next let $m$ be a positive integer such that $1\leq m \leq n$, and fix a set $X_{n+m} := \{x_1,x_2,\ldots, x_{n+m}\}$. Fix $\a_1,\ldots, \a_m \in (0,1]$. For each $0 \leq t \leq 1$, define the matrix $\d_t:= \matele{d^t_{ij}}_{i,j = 1}^{n+m}$ by:
\[\text{For } 1 \leq i,j \leq n+m, \qquad d^t_{ij}:= 
\begin{cases}
0 &: i = j\\
f(\a_i,t) &: j - i = n \\
f(\a_j,t) &: i - j = n\\
t &: \text{ otherwise.}
\end{cases}
\]
This is a block matrix $\mattwo{A}{B}{B^T}{C}$ where $A$ is the $n\times n$ all-$t$ matrix with 0s on the diagonal, $C$ is an $m\times m$ all-$t$ matrix with 0s on the diagonal, and $B$ is the $n\times m$ all-$t$ matrix with $f(\a_1,t), f(\a_2,t),\ldots, f(\a_m,t)$ on the diagonal.

We first claim that $\d_t$ is the distance matrix of a pseudometric space. Symmetry is clear. We now check the triangle inequality. 
In the cases $1\leq i,j,k \leq n$ and $n+1 \leq i,j,k \leq n+m$, the points $x_i,x_j,x_k$ form the vertices of an equilateral triangle with side length $t$. Suppose $1\leq i,j \leq n$ and $n+1 \leq k \leq n+m$. 
Then the triple $x_i,x_j,x_k$ forms an isosceles triangle with equal longest sides of length $t$, and a possibly shorter side of length $f(\a_i,t)$ (if $|k-i| = n$), $f(\a_j,t)$ (if $|k-j| = n$), or just a third equal side with length $t$ in the remaining cases. 
The case $1\leq i \leq n$, $n+1 \leq j,k \leq n+m$ is similar. This verifies the triangle inequality. 
Also note that $\d_t$ is the distance matrix of a bona fide metric space for $ t\in (0,1)$. 
For $t = 1$, we identify the points $x_i$ and $x_{i-n}$, for $n+1 \leq i \leq n+m$, to obtain $\Delta_n$, and for $t= 0$, we identify all points together to obtain $\Delta_1$. 
This allows us to define geodesics between $\Delta_1$ and $\Delta_n$ as follows. Let $\vec{\alpha}$ denote the vector $(\a_1,\ldots, \a_m)$. We define a map $\g_{\vec{\a}}: [0,1] \r \M$ by writing:
\[\g_{\vec{\a}}(t) := (X_{n+m},\d_t) \qquad t \in [0,1],\]
where we can take quotients at the endpoints as described above.

We now verify that these curves are indeed geodesics. There are three cases: $s,t \in [0,\tfrac{1}{2}]$, $s,t \in (\tfrac{1}{2},1]$, and $s \in [0,\tfrac{1}{2}]$, $t\in (\tfrac{1}{2},1]$. By using the diagonal correspondence $\diag$, we check case-by-case that $\dis(\diag) \leq |t-s|$. 
%
%
%
Thus for any $s,t \in [0,1]$, we have $\dgh(\g_{\vec{\a}}(s),\g_{\vec{\a}}(t) ) \leq \tfrac{1}{2}|t-s| = |t-s|\cdot \dgh(\Delta_1,\Delta_n)$. It follows by Lemma \ref{lemma:curves} that $\g_{\vec{\a}}$ is a geodesic between $\Delta_1$ and $\Delta_n$. Furthermore, since $\vec{\a} \in (0,1]^m$ was arbitrary, this holds for any such $\vec{\a}$. Thus we have an infinite family of geodesics $\g_{\vec{\a}}:[0,1] \r \M$ from $\Delta_1$ to $\Delta_n$.

A priori, some of these geodesics may intersect at points other than the endpoints. By this we mean that there may exist $t \in (0,1)$ and $\vec{\a} \neq \vec{\b} \in (0,1]^m$ such that $[\g_{\vec{\a}}(t)] = [\g_{\vec{\b}}(t)]$ in $\Msim$. This is related to the branching phenomena that we describe in the next section. For now, we give an infinite subfamily of geodesics that do not intersect each other anywhere except at the endpoints. Recall that the \emph{separation} of a finite metric space $(X,d_X)$ is the smallest positive distance in $X$, which we denote by $\sep(X)$. If $\sep(X)< \sep(Y)$ for two finite metric spaces $X$ and $Y$, then $\dgh(X,Y) > 0$. 

Let $\prec$ denote the following relation on $(0,1]^m$: for $\vec{\a}, \vec{\b} \in (0,1]^m$, set $\vec{\a} \prec \vec{\b}$ if $\a_i < \b_i$ for each $1\leq i \leq m$. Next let $\vec{\a}, \vec{\b} \in (0,1]^m$ be such that $\vec{\a} \prec \vec{\b}$. 
Then $\g_{\vec{\b}}$ is a geodesic from $\Delta_1$ to $\Delta_n$ which is distinct (i.e. non-isometric) from $\g_{\vec{\a}}$ everywhere except at its endpoints. 
This is because the condition $\vec{\a} \prec \vec{\b}$ guarantees that for each $t \in (0,1)$, $\sep(\g_{\vec{\a}}(t)) < \sep(\g_{\vec{\b}}(t))$. 
Hence $\dgh(\g_{\vec{\a}}(t),\g_{\vec{\b}}(t)) > 0$ for all $t\in (0,1)$.

Finally, let $\vec{\a} \in (0,1)^m$, and let $\vec{1}$ denote the all-ones vector of length $m$. For $\eta \in [0,1]$, define $\vec{\b}(\eta) := (1-\eta)\vec{a} + \eta\vec{1}$. Then by the observations about the relation $\prec$, $\{\g_{\vec{\b}(\eta)} : \eta \in [0,1]\}$ is an infinite family of geodesics from $\Delta_1$ to $\Delta_n$ that do not intersect pairwise anywhere except at the endpoints.

Note that one could choose the diameter of $\Delta_n$ to be arbitrarily small and still obtain deviant geodesics via the construction above.

\subsubsection{Branching}
\label{sec:branching}

\begin{wrapfigure}{R}{0.2\linewidth}
\includegraphics[scale=0.7]{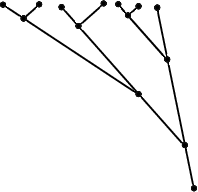}
\end{wrapfigure}

The structure of $\dgh$ permits \emph{branching geodesics}, as illustrated on the right. We use the notation $(a)^+$ for any $a\in \R$ to denote $\max(0,a)$. As above, fix $n\in \N$, $n \geq 2$, and consider the straight-line geodesic between $\Delta_1$ and $\Delta_n$ described at the beginning of Section \ref{sec:deviant}. Throughout this section, we denote this geodesic by $\g:[0,1] \r \M$. We will construct an infinite family of geodesics which branch off from $\g$. For convenience, we will overload notation and write, for each $t\in [0,1]$, the distance matrix of $\g(t)$ as $\g(t)$. Recall from above that $\g(t)$ is a symmetric $n\times n$ matrix with the following form:
\[
\begin{pmatrix}
0 & t & t & \dots  & t \\
  & 0 & t & \dots  & t \\
  &   & \ddots & \dots  & \vdots \\ 
  &   &   &        & 0
\end{pmatrix}
\]

Let $(a_i)_{i\in \N}$ be any sequence such that $0 < a_1 < a_2 < \ldots < 1$. For each $t \in [0,1]$, define the $(n+1)\times (n+1)$ matrix $\g^{(a_1)}(t)$ to be the symmetric matrix with the following upper triangular form:
\[ 
\lp \g^{(a_1)}(t) \rp_{ij}:=
\begin{cases}
(\g(t))_{ij} &: 1\leq i\leq j \leq n\\
(\g(t))_{in} &: 1\leq i < n, \; j = n+1\\
(t-a_1)^+ &: i = n, \; j = n+1\\
0 &: i = n+1, \; j = n+1
\end{cases}
\;
\vline
\;
\begin{pmatrix}
0 & t & \dots  & t  & t\\
  & 0 &  \dots  & t  & t\\
  &   & \ddots & \vdots & \vdots \\ 
 &    &        & 0 &(t-a_1)^+\\
  &     &        &  &0
\end{pmatrix}
\]

For $t > a_1$, we have $\dgh\lp\g(t), \g^{(a_1)}(t) \rp > 0$, because any correspondence between $\g(t), \g^{(a_1)}(t)$ has distortion at least $t-a_1$. Thus $\g^{(a_1)}$ branches off from $\g$ at $a_1$.

The construction of $\g^{(a_1)}(t)$ above is a special case of a \emph{one-point metric extension}. Such a construction involves appending an extra row and column to the distance matrix of the starting space; explicit conditions for the entries of the new row and column are stated in \cite[Lemma 5.1.22]{pestov2006dynamics}. In particular, $\g^{(a_1)}(t)$ above satisfies these conditions.

Procedurally, the $\g^{(a_1)}(t)$ construction can be generalized as follows. Let $(\bullet)$ denote any finite subsequence of $(a_i)_{i\in \N}$. 
We also allow $(\bullet)$ to be the empty subsequence. Let $a_j$ denote the terminal element in this subsequence. Then for any $a_k$, $k > j$, we can construct $\g^{(\bullet,a_k)}$ as follows:
\begin{enumerate}
\item Take the rightmost column of $\g^{(\bullet)}(t)$, replace the only 0 by $(t-a_k)^+$, append a 0 at the bottom.
\item Append this column on the right to a copy of $\g^{(\bullet)}(t)$. 
\item Append the transpose of another copy of this column to the bottom of the newly constructed matrix to make it symmetric.
\end{enumerate}

The objects produced by this construction satisfy the one-point metric extension conditions \cite[Lemma 5.1.22]{pestov2006dynamics}, and hence are distance matrices of pseudometric spaces. By taking the appropriate quotients, we obtain valid distance matrices. Symmetry is satisfied by definition, and the triangle inequality is satisfied because any triple of points forms an isosceles triangle with longest sides equal. We write $\Gamma^{(\bullet)}(t)$ to denote the matrix obtained from $\g^{(\bullet)}(t)$ after taking quotients. As an example, we obtain the following matrices after taking quotients for $\g^{(a_1)}(t)$ above, for $0\leq t \leq a_1$ (below left) and for $a_1 < t \leq 1$ (below right):
\[
\begin{pmatrix}
0 & t & \dots  & t  \\
  & 0 &  \dots  & t  \\
  &   & \ddots & \vdots \\ 
 &    &        & 0 
\end{pmatrix}
\hspace{0.25 in}
\vline
\hspace{0.25 in}
\begin{pmatrix}
0 & t & \dots  & t  & t\\
  & 0 &  \dots  & t  & t\\
  &   & \ddots & \vdots & \vdots \\ 
 &    &        & 0 &(t-a_1)\\
  &     &        &  &0
\end{pmatrix}
\]

Now let $(a_{i_j})_{j = 1}^k$ be any finite subsequence of $(a_i)_{i \in \N}$. For notational convenience, we write $(b_i)_{i}$ instead of $(a_{i_j})_{j = 1}^k$. $\Gamma^{(b_i)_i}$ is a curve in $\M$; we need to check that it is moreover a geodesic.


Let $s \leq t \in [0,1]$. Then $\Gamma^{(b_i)_i}(s)$ and $\Gamma^{(b_i)_i}(t)$ are square matrices with $n+p$ and $n+q$ columns, respectively, for nonnegative integers $p$ and $q$. It is possible that the matrix grows in size between $s$ and $t$, so we have $q \geq p$. Denote the underlying point set by $\{x_1,x_2,\ldots, x_{n+p},\ldots, x_{n+q}\}$. Then define:
\[A:= \{(x_i,x_i) : 1\leq i \leq n+p\}, \;
B:= \{(x_{n+p},x_j) : n+p <j \leq n+q\},\; 
R:= A\cup B.\]

Here $B$ is possibly empty. Note that $R$ is a correspondence between $\Gamma^{(b_i)_i}(s)$ and $\Gamma^{(b_i)_i}(t)$, and by direct calculation we have $\dis(R) \leq |t-s|$.
Hence we have $\dgh\lp\Gamma^{(b_i)_i}(s), \Gamma^{(b_i)_i}(t) \rp \leq \tfrac{1}{2}\cdot |t-s| = |t-s|\cdot \dgh(\Delta_1,\Delta_n)$. An application of Lemma \ref{lemma:curves} now shows that $\Gamma^{(b_i)_i}$ is a geodesic.

The finite subsequence $(b_i)_i$ of $(a_i)_{i\in \N}$ was arbitrary. Thus we have an infinite family of geodesics which branch off from $\g$. Since the increasing sequence $(a_i)_{i\in\N}\in (0,1)^\N$ was arbitrary, the branching could occur at arbitrarily many points along $\g$.

\begin{remark}
The existence of branching geodesics shows that $(\Msim,\dgh)$ is \emph{not} an Alexandrov space with curvature bounded below \cite[Chapter 10]{burago}. Moreover, the existence of deviant (i.e. non-unique) geodesics shows that $(\Msim,\dgh)$ cannot have curvature bounded from above, i.e. $(\Msim,\dgh)$ is not a CAT($k$) space for any $k > 0$ \cite[Proposition 2.11]{bridson2011metric}. 
\end{remark}

\subsection{An explicit geodesic from $\sph^0$ to $\sph^n$}\label{sec:spheres}
Let $n \in \N$. Consider the spheres $\sph^0$ and $\sph^n$ equipped with the canonical geodesic metric, such that each sphere has diameter $\pi$. We will now construct an explicit geodesic between $\sph^0$ and $\sph^n$. 

\begin{proposition}\label{prop:dgh-spheres} $\dgh(\sph^0,\sph^n) = \tfrac{\pi}{2}$.
\end{proposition}
\begin{proof}[Proof of Proposition \ref{prop:dgh-spheres}]
Let $U$ and $L$ denote the closed upper hemisphere and open lower hemisphere of $\sph^n$, respectively. Then we have $U \sqcup L = \sph^n$. Moreover, let $s,s'\in U$ be two points realizing the diameter of $\sph^n$ via an arc in $U$ (note that such an arc exists for each $\sph^n$ when $n\geq 1$). Also let $\set{p,q}$ denote the two points of $\sph^0$. Now we construct a correspondence between $\sph^0$ and $\sph^n$:
\[R:=\set{(p,u) : u \in U} \cup \set{(q,l) : l\in L}.\]
Then we have:
\begin{align*}
\dis(R) = \sup_{(x,y),(x',y')\in R}|d_{\sph^n}(y,y') - d_{\sph^0}(x,x')| = |d_{\sph^n}(s,s') - d_{\sph^0}(p,p)| = \pi.
\end{align*}
It follows that $\dgh(\sph^0,\sph^n) \leq \tfrac{\pi}{2}$. Next we wish to show the reverse inequality. Let $T$ be an arbitrary correspondence between $\sph^0$ and $\sph^n$, and write:
\[P:=\set{x \in \sph^n : (p,x) \in T}, \qquad
Q:=\set{x \in \sph^n : (q,x) \in T}.\]
Then, $P\neq \emptyset$, $Q\neq \emptyset$, and $\sph^n = P \cup Q$, by the definition of correspondences. Now recall the Lusternik-Schnirelmann theorem (\cite[p. 117]{bollobas2006art}, also see \cite[p. 33]{hatcher}): for every family of $n+1$ closed sets covering $\sph^n$, one of the sets contains a pair of antipodal points. Applying this result by taking $n$ copies of $\overline{P}$ and one copy of $\overline{Q}$ as the cover of $\sph^n$, we get that at least one of the sets $\overline{P}$ and $\overline{Q}$ contains a pair of antipodal points. Without loss of generality, suppose $\overline{P}$ contains a pair of antipodal points $(a,a')$. Let $(a_n), (a'_n)$ be sequences in $P$ such that for each $n\in \N$, we have $a_n \in B(a,\tfrac{1}{n})$ and $a'_n \in B(a',\tfrac{1}{n})$. 
By the triangle inequality, one has that $|d_{\sph^n}(a_n,a'_n) - d_{\sph^n}(a,a')| \leq d_{\sph^n}(a_n,a) + d_{\sph^n}(a'_n,a') < \tfrac{2}{n}$. Also note that $|d_{\sph^n}(a,a') - d_{\sph^0}(p,p)| = \pi$. Then we obtain: 
\[|d_{\sph^n}(a_n,a'_n) - d_{\sph^0}(p,p)| = |d_{\sph^n}(a_n,a'_n) - d_{\sph^n}(a,a') + d_{\sph^n}(a,a') - d_{\sph^0}(p,p)| > \pi - \tfrac{2}{n}.\] 
By letting $n \r \infty$, it follows that $\dis(T) \geq \pi$. Since $T$ was an arbitrary correspondence, we obtain $\dgh(\sph^0,\sph^n) \geq  \tfrac{\pi}{2}$. Thus we obtain $\dgh(\sph^0,\sph^n) =  \tfrac{\pi}{2}$. \end{proof}

It now follows that the correspondence $R$ defined in the proof of Proposition \ref{prop:dgh-spheres} is an optimal correspondence between $\sph^0$ and $\sph^n$. In particular, the definition of $R$ suggests that one may define a geodesic from $\sph^n$ to $\sph^0$ by ``shrinking" $U$ and $L$ to the north and south poles, respectively. 


\begin{proposition}\label{prop:s0-sn-geod} Let $U$ and $L$ denote the closed upper hemisphere and open lower hemisphere of $\sph^n$, respectively. Also let $\set{p,q}$ denote the two points of $\sph^0$, and let $\mu\in U$, $ \lambda\in L$ denote the north and south poles of $\sph^n$, respectively. For each $t \in (0,1)$, define:
\begin{align*}
&U_t:=U \cap \overline{B(\mu,(1-t)\cdot \tfrac{\pi}{2})}, \qquad
L_t:=L \cap \overline{B(\lambda,(1-t)\cdot \tfrac{\pi}{2})}\\
&\hspace{0.1\textwidth} X_t:= U_t \cup L_t, \qquad
d_{X_t}:=d_{\sph^n|X_t \times X_t}.
\end{align*}
Finally define $\g: [0,1] \r \M$ by $\g(0):=(\sph^0,d_{\sph^0})$, $\g(1):=(\sph^n,d_{\sph^n})$, and 
\[\g(t):= (X_t, d_{X_t}) \text{ for each } t\in (0,1).\]
Then $\g$ is a geodesic from $\sph^0$ to $\sph^n$.
\end{proposition}

\begin{proof}[Proof of Proposition \ref{prop:s0-sn-geod}]
First let $t\in (0,1)$, and define a correspondence between $\sph^0$ and $X_t$ by:
\[R_t:=\set{(p,u) : u \in U_t} \cup \set{(q,l) : l \in L_t}.\]
Then we have $\dis(R_t)= (1-t) \pi,$ and so $\dgh(\sph^0, \g(t)) \leq \tfrac{(1-t)\pi}{2}.$ Similarly we obtain $\dgh(\g(t),\sph^n) \leq \tfrac{t\pi}{2}$. We wish to show that these inequalities are actually equalities. Without loss of generality, suppose that $\dgh(\sph^0,\g(t)) < \tfrac{(1-t)\pi}{2}$. Then we obtain:
\[\dgh(\sph^0,\sph^n) \leq \dgh(\sph^0,\g(t)) + \dgh(\g(t),\sph^n) < \tfrac{\pi}{2}.\]
This is a contradiction, by Proposition \ref{prop:dgh-spheres}. Thus for each $t\in (0,1)$, we have $\dgh(\sph^0,\g(t)) = \tfrac{(1-t)\pi}{2}$ and $\dgh(\sph^n,\g(t)) = \tfrac{t\pi}{2}$. 

Next let $s\in (0,1)$. We wish to show $\dgh(\g(s),\g(t)) = |t-s|\cdot \dgh(\sph^0,\sph^n)$. We have two cases: (1) $s \leq t$, and (2) $s > t$. Both cases are similar, so we just show the first case. Before proceeding, notice that since $s \leq t$, we have $U_t \subseteq U_s$ and $L_t \subseteq L_s$, and so $X_t \subseteq X_s$. 

We will define some notation for convenience. For each $x \in U_s$ let $c_x^{\mu}$ denote the shortest geodesic segment connecting $x$ to the north pole $\mu$. Next, for each $x \in L_s$ let $c_x^{\lambda}$ denote the shortest geodesic segment connecting $x$ to the south pole $\lambda$. Also write $\operatorname{bd}(U_t)$ and $\operatorname{bd}(L_t)$ to denote the boundaries of $U_t$ and $L_t$.

Now define a map $\ph^U:U_s \r U_t$ by: 
\[\ph^U(x):=\begin{cases}
x &, \; x \in U_t\\
c_x^{\mu}\cap \operatorname{bd}(U_t) &,\; x \in U_s\setminus U_t.
\end{cases}
\]
Also define a map $\ph^L:L_s \r L_t$ by: 
\[\ph^L(x):=\begin{cases}
x &, \; x \in L_t\\
c_x^{\lambda}\cap \operatorname{bd}(L_t) &,\; x \in L_s\setminus L_t.
\end{cases}
\]

Finally define $\ph:X_s \r X_t$ as follows:
\[\ph(x):=\begin{cases}
\ph^U(x) &,\; x \in U_s\\
\ph^L(x) &,\; x \in L_s.
\end{cases}
\]

Observe that for any $x\in X_s$ we have:  
\[d_{\sph^n}(x,\ph(x)) \leq (1-s)\cdot \tfrac{\pi}{2} - (1-t)\cdot \tfrac{\pi}{2} = (t-s)\cdot\tfrac{\pi}{2}.\]

Now define $T:=\set{(x,\ph(x)): x \in X_s}$. This is a correspondence between $X_s$ and $X_t$. By the preceding calculation, we have:
\begin{align*}
\dis(T) &= \sup_{x,x'\in X_s}|d_{\sph^n}(x,x') - d_{\sph^n}(\ph(x),\ph(x'))| \\
&\leq \sup_{x,x'\in X_s}\big(d_{\sph^n}(x,\ph(x)) + d_{\sph^n}(x',\ph(x'))\big) \leq (t-s)\cdot \pi.
\end{align*}
Thus $\dgh(\g(s),\g(t)) \leq (t-s)\cdot \tfrac{\pi}{2}$. Similarly, we obtain $\dgh(\g(s),\g(t)) \leq (s-t)\cdot \tfrac{\pi}{2}$ when $s > t$. Thus we obtain:
\[\dgh(\g(s),\g(t)) \leq |t-s|\cdot \tfrac{\pi}{2} = |t-s|\cdot \dgh(\sph^0,\sph^n).\]
This inequality must be an equality by Lemma \ref{lemma:curves}, which completes the proof that $\g$ is a geodesic.
 \end{proof}

\section{Proof of Theorem \ref{thm:geodesic}}\label{sec:proofs}

This section contains our proof showing that $(\Msim,\dgh)$ is a geodesic space. This fact was established in \cite{ivanov2015gromov} via an application of Gromov's precompactness theorem. In this paper we present a different, direct proof of this fact based on ideas used by Sturm in the setting of metric measure spaces \cite{sturm2012space}. In particular, our method of proof provides an explicit construction of geodesics. 


Let $X$ and $Y$ be compact metric spaces. Endow $X\times Y$ with the product metric \[\delta\big((x,y),(x',y')\big):=\max\big(d_X(x,x'),d_Y(y,y')\big),\,\,\mbox{for all $(x,y),(x',y')\in X\times Y$}.\]
Note that $X\times Y$ is compact. Next consider the set of all non-empty closed subsets of $X\times Y$, denoted $\mathcal{C}(X\times Y)$, endowed with the Hausdorff distance $\d_{\mathrm{H}}$ arising from $\d$. It follows from Blaschke's theorem \cite{burago} that $\mathcal{C}(X\times Y)$ is also compact.

\begin{lemma}\label{lem:dist} Let $X \times Y$ be the compact metric space with product metric $\d$ as defined above. Let $R,S\subset X\times Y$ be any two non-empty relations. Then,
\begin{enumerate}
\item $d_\mathrm{H}^X(\pi_1(R),\pi_1(S))\leq \d_\mathrm{H}(R,S).$\label{eq:dist1}
\item $d_\mathrm{H}^Y(\pi_2(R),\pi_2(S))\leq \d_\mathrm{H}(R,S).$\label{eq:dist2}
\item $|\dis(R)-\dis(S)|\leq 4\, \d_\mathrm{H}(R,S).$\label{eq:dist3}
\end{enumerate}
Here $\pi_1$ and $\pi_2$ are the natural projections of $X\times Y$ onto $X$ and $Y$, respectively.
\end{lemma}

\begin{proof} To show (\ref{eq:dist1}), let $\eta > \dh(R,S)$. Let $x \in \pi_1(R)$, and let $y \in Y$ be such that $(x,y) \in R$. Then there exists $(x',y') \in S$ such that $\d\left((x,y),(x',y')\right) < \eta$. Thus $d_X(x,x') < \eta,$ where $x' \in \pi_1(S)$. Similarly, given any $u \in \pi_1(S)$, we can find $u' \in \pi_1(R)$ such that $d_X(u,u') < \eta$. Thus $d_{\operatorname{H}}^X(\pi_1(R),\pi_1(S))< \eta$. Since $\eta > \dh(R,S)$ was arbitrary, it follows that $d_{\operatorname{H}}^X(\pi_1(R),\pi_1(S))\leq \dh(R,S).$

The proof for inequality (\ref{eq:dist2}) is similar, so we omit it.

To prove inequality (\ref{eq:dist3}), let $\eta > \dh(R,S)$, and let $\e\in (\dh(R,S),\eta)$. Define 
$$L:=\set{(r,s)\in R\times S : \d(r,s) < \e}.$$
Note that, since $\d_H(R,S) < \e$, $L$ is a correspondence between $R$ and $S$.  Then,

\begin{align*}
&\hspace{0.4\textwidth}|\dis(R) - \dis(S)|\\
&=\bigg|\sup_{(x,y),(x',y')\in R}|d_X(x,x')-d_Y(y,y')|-
\sup_{(u,v),(u',v')\in S}|d_X(u,u')-d_Y(v,v')|\bigg|\\
&\leq \sup_{\substack{\left((x,y),(u,v)\right),\\ \left((x',y'),(u',v')\right) \in L}}\big| |d_X(x,x')-d_Y(y,y')|-|d_X(u,u')-d_Y(v,v')|\big|\\
&\leq \sup_{\substack{\left((x,y),(u,v)\right),\\ \left((x',y'),(u',v')\right) \in L}}\big(|d_X(x,x')-d_X(u,u')|+|d_Y(v,v')-d_Y(y,y')|\big)\\
&\leq \sup_{\substack{\left((x,y),(u,v)\right),\\ \left((x',y'),(u',v')\right) \in L}}\big(d_X(x,u)+d_X(x',u') + d_Y(v,y) + d_Y(v',y')\big) \;\text{(Triangle ineq.)}\\
&\leq \sup_{\substack{\left((x,y),(u,v)\right),\\ \left((x',y'),(u',v')\right) \in L}}\big(2\d((x,y),(u,v)) + 2\d((x',y'),(u',v'))\big)\\
&\leq 4\e < 4\eta.
\end{align*}
But $\eta > \dh(R,S)$ was arbitrary. It follows that $|\dis(R)-\dis(S)| \leq 4\,\dh(R,S)$.
\end{proof}


\begin{proof}[Proof of Proposition \ref{prop:optimal}] Let $(\e_n)_{n} \downarrow 0 $ be an arbitrary sequence in $\R_+$. For each $n$, let $X_n, Y_n$ be $\e_n$-nets for $X$ and $Y$ respectively. It is a fact that if $S$ is an $\e$-net in a metric space $X$, then $\dgh(S,X) < \e$ (\cite[Example 7.3.11]{burago}). Thus $\dgh(X_n,X) \r 0$ and $\dgh(Y_n,Y) \r 0$ as $n \r \infty$. Optimal correspondences always exist between finite metric spaces, so for each $n\in \N$, let $R_n \in \Rsc(X_n,Y_n)$ be such that $\dis(R_n)=2\,\dgh(X_n,Y_n)$. Since $(R_n)_n$ is a sequence in the compact metric space $\mathcal{C}(X\times Y)$, it contains a convergent subsequence. To avoid double subscripts, we reindex if necessary and let $(R_n)_n$ denote this convergent subsequence. Let $R \in \mathcal{C}(X\times Y)$ denote the $\dh$-limit of $(R_n)_n$, i.e. $\ds\lim_{n\r \infty}\dh(R_n,R) =0$. Then by Lemma \ref{lem:dist},

$$|\dis(R_n)-\dis(R)|\leq 4\dh(R_n,R) \r 0\text{ as $n \r \infty$}.$$

So $\dis(R_n) \r \dis(R)$. But also,
$$\dis(R_n) = 2\,\dgh(X_n,Y_n) \r 2\,\dgh(X,Y) \text{ as $n\r \infty$},$$
since $|\dgh(X,Y) - \dgh(X_n,Y_n)| \leq \dgh(X_n,X)+\dgh(Y_n,Y) \r 0$. Then we have:
$$\dis(R) = 2\,\dgh(X,Y).$$

It remains to show that $R$ is a correspondence. Note that for any $n$,
\begin{align*}
d_{\operatorname{H}}^X(X,\pi_1(R)) &\leq d_{\operatorname{H}}^X(X,\pi_1(R_n))+
d_{\operatorname{H}}^X(\pi_1(R_n),\pi_1(R))\\
&\leq d_{\operatorname{H}}^X(X,X_n) + \dh(R_n,R) \qquad\text{(By Lemma \ref{lem:dist}).}\\
\end{align*}
But the term on the right can be made arbitrarily small, since each $X_n$ is an $\e_n$-net for $X$ and $\lim_{n\r \infty}\dh(R_n,R)=0$. Thus $d_{\operatorname{H}}^X(X,\pi_1(R)) = 0$, and therefore $X=\overline{\pi_1(R)}$. Since $R$ is a closed subset of the compact space $X\times Y$, it is compact, and its continuous image $\pi_1(R)$ is also compact, hence closed (since $X$ is Hausdorff). Thus $\overline{\pi_1(R)}=\pi_1(R)=X$. Similarly, it can be shown that $\pi_2(R)=Y$. Thus $R$ is a correspondence. This concludes the proof.\end{proof}

\begin{proof}[Proof of Theorem \ref{thm:geodesic}] Suppose we can find a curve $\g: [0,1] \r \M$ such that $\g(0)=(X,d_X)$ and $\g(1)=(Y,d_Y)$, and for all $s, t \in [0,1]$, 
\[\dgh(\g(s),\g(t)) = |t-s|\cdot \dgh(X,Y).\]
Then we also have $\dgh([\g(s)],[\g(t)]) = |t-s|\cdot \dgh([X],[Y])$ for all $s,t \in [0,1]$, and we will be done. So we will show the existence of such a curve $\g$. 

Let $R\in \Ropt(X,Y),$ i.e. let $R$ be a correspondence between $X$ and $Y$ such that $\dis(R)=2\,\dgh(X,Y)$. Such a correspondence always exists by Proposition \ref{prop:optimal}.

For each $t\in(0,1)$ define $\g(t)= \big(R,d_{\gamma(t)}\big)$ where 
\[d_{\gamma(t)}\big((x,y),(x',y')\big)=(1-t)\cdot d_X(x,x')+t\cdot d_Y(y,y')\]
for all $(x,y),(x',y')\in R$. Note that for each $t\in(0,1)$ $d_{\g(t)}$ is a legitimate metric on $R$. We also set $\g(0) = (X,d_X)$ and $\g(1) = (Y,d_Y)$.

\begin{claim}\label{cl:ineq} For any $s,t \in [0,1]$, 
$$\dgh(\g(s),\g(t)) \leq |t-s|\cdot \dgh(X,Y).$$
\end{claim}
Suppose for now that Claim \ref{cl:ineq} holds. 
Lemma \ref{lemma:curves} implies that, for all $s, t \in [0,1]$, 
$\dgh(\g(s),\g(t)) = |t-s|\cdot \dgh(X,Y).$ Thus it suffices to show Claim \ref{cl:ineq}. There are three cases: (i) $s,t \in (0,1)$, (ii) $s=0, t \in (0,1)$, and (iii) $s\in (0,1), t=1$. The last two cases are similar, so we just prove (i) and (ii).

For (i), fix $s, t \in (0,1)$. 
Taking the diagonal correspondence $\diag\in \Rsc(R,R)$, we get:

\begin{align*}
\dis(\diag) &= \sup_{(a,a),(b,b) \in \diag}\vert d_{\g(t)}(a,b) - d_{\g(s)}(a,b)\vert\\
&= \sup_{(x,y),(x',y') \in R}\vert d_{\g(t)}((x,y),(x',y')) - d_{\g(s)}((x,y),(x',y'))\vert\\
&= \sup_{(x,y),(x',y') \in R}\vert (1-t)\cdot d_X(x,x') + t \cdot d_Y(y,y') \\
& \hspace{1 in}
- (1-s)\cdot d_X(x,x') -s\cdot d_Y(y,y') \vert\\
&= \sup_{(x,y),(x',y') \in R}\vert (s-t)\cdot d_X(x,x') - (s-t)\cdot d_Y(y,y')\vert\\
&= |t-s|\cdot\sup_{(x,y),(x',y') \in R}\vert d_X(x,x') - d_Y(y,y')\vert\\
&= 2|t-s|\cdot \dgh(X,Y).
\end{align*}
Finally $\dgh(\g(t),\g(s))\leq \frac{1}{2}\dis(\diag) = |t-s|\cdot\dgh(X,Y)$. This proves case (i) of Claim \ref{cl:ineq}.

For (ii), fix $s=0, t\in (0,1)$. Define $R_X = \set{(x,(x,y)) : (x,y) \in R}$. Then $R_X$ is a correspondence in $\Rsc(X, R)$.
\begin{align*}
\dis(R_X) &= \sup_{(x,(x,y)),(x',(x',y')) \in R_X}|d_X(x,x') - (1-t)\cdot d_X(x,x') - t\cdot d_Y(y,y')|\\
&=\sup_{(x,(x,y)),(x',(x',y')) \in R_X}|d_X(x,x') - d_Y(y,y')|\cdot t\\
&=t\cdot \dis(R) = 2t\cdot\dgh(X,Y).
\end{align*}
Thus $\dgh(X,\g(t)) \leq t\cdot \dgh(X,Y)$. The proof for case (iii), i.e. that $\dgh(\g(s),Y) \leq |1-s|\cdot \dgh(X,Y)$, is similar. Thus Claim \ref{cl:ineq} follows. 
The theorem now follows.\end{proof}

\section{Discussion}

While we provide an explicit construction of straight-line geodesics, it is natural to ask the following: can we characterize other classes of geodesics in $(\Msim,\dgh)$? In Section \ref{sec:deviant}, we constructed infinite families of deviant (i.e. non-unique) geodesics between $\Delta_1$ and $\Delta_n$. In Section \ref{sec:branching}, we provided a parametric construction by which the straight-line geodesic between $\Delta_1$ and $\Delta_n$ could be made to branch off into arbitrarily many nodes at arbitrarily many locations. 

As stated at the end of Section \ref{sec:geods}, the existence of branching and deviant geodesics shows the negative result that $(\Msim,\dgh)$ cannot have curvature bounded from above or below. In light of this result, it is interesting to point out the work of Sturm showing that the space of metric \emph{measure} spaces \cite{sturm2012space} has nonnegative curvature when equipped with an $L^2$-Gromov-Wasserstein metric.

\section{Acknowledgments}
We thank Prof.~Vladimir Zolotov for pointing out the existence of branching geodesics to us. We also thank the referees for their helpful comments.

\bibliographystyle{amsplain}
\bibliography{biblio}
\end{document}